\documentclass[12pt]{amsart}
\usepackage{graphicx}
\graphicspath{ {./images/} }
\usepackage{float}
\usepackage{amsmath, amsthm, amsfonts, amssymb} %ams packages
\usepackage{mathrsfs} %script letters
\usepackage{bbm} %additional blackboard bold characters
\usepackage{txfonts}
\usepackage[T1]{fontenc}
\usepackage{enumerate}
\usepackage{comment}
\usepackage{hyperref}
\usepackage[capitalise]{cleveref}
\usepackage{xcolor} %allow colorful text
\usepackage{chngcntr} %numbers tables by section
\usepackage{bm}
\counterwithin{table}{section} %numbers tables by section
\theoremstyle{plain}
\newtheorem{thm}{Theorem}[section]

\newtheorem{lem}[thm]{Lemma}
\newtheorem{prop}[thm]{Proposition}
\newtheorem*{prob*}{Problem}
\newtheorem*{ques*}{Question}
\newtheorem*{thm*}{Theorem}

\theoremstyle{definition}
\newtheorem{defn}[thm]{Definition}

\newtheorem{conj}[thm]{Conjecture}
\newtheorem*{defn*}{Definition}
\newtheorem{rem}[thm]{Remark}

\newtheorem{rem*}[thm]{Remark}
\newtheorem{question}[thm]{Question}
\numberwithin{equation}{section}

\newcommand{\Z}{\mathbb Z}
\newcommand{\N}{\mathbb N}
\newcommand{\Q}{\mathbb Q}

\newcommand{\T}{\mathbb{T}}

\newcommand{\Hc}{\mathcal{H}}

\newcommand{\innprod}[2]{\left\langle #1, #2 \right\rangle}

\renewcommand{\tilde}{\widetilde}
\renewcommand{\hat}{\widehat}

\usepackage[normalem]{ulem}

\usepackage{cancel}

	\title[On the maximal spectral type of Nilsystems]{On the maximal spectral type of Nilsystems}
	
    \author[E. Ackelsberg]{Ethan Ackelsberg}
\address{Institute for Advanced Study \\ 1 Einstein Drive \\ 
Princeton, New Jersey \\
08540, USA}
\email{eackelsberg@ias.edu}
	
	\author[F. K. Richter]{Florian K. Richter}
	\address{\'{E}cole Polytechnique F\'{e}d\'{e}rale de Lausanne (EPFL)\\ Lausanne, Vaud, Switzerland}
	\email{f.richter@epfl.ch}
	
\author[O. Shalom]{Or Shalom}
\address{Institute for Advanced Study \\ 1 Einstein Drive \\ 
Princeton, New Jersey \\
08540, USA}
\email{Or.Shalom@ias.edu}
	
\date{\today}

\begin{document}
	\begin{abstract}
	Let $(G/\Gamma,R_a)$ be an ergodic $k$-step nilsystem for $k\geq 2$. We adapt an argument of Parry \cite{Parry} to show that $L^2(G/\Gamma)$ decomposes as a sum of a subspace with discrete spectrum and a subspace of Lebesgue spectrum with infinite multiplicity. In particular, we generalize a result previously established by Host--Kra--Maass \cite{HKM} for $2$-step nilsystems and a result by Stepin \cite{Stepin} for nilsystems $G/\Gamma$ with connected, simply connected $G$.
    \end{abstract}
    \maketitle
    \tableofcontents
    \section{Introduction}
    A \emph{nilmanifold} is a compact manifold of the form $G/\Gamma$ where $G$ is a nilpotent Lie group and $\Gamma$ is a discrete co-compact subgroup. If $G$ is $k$-step nilpotent, then we say that $G/\Gamma$ is a \emph{$k$-step nilmanifold}.
    For any $a\in G$ let $R_a\colon G/\Gamma\rightarrow G/\Gamma$ denote the left-translation by $a$ on $G/\Gamma$, that is, $R_a (g \Gamma) = (a g) \Gamma$ for all $g\Gamma\in G/\Gamma$. The resulting topological dynamical system $(G/\Gamma, R_a)$ is called a \emph{nilsystem}. Every nilmanifold $G/\Gamma$ admits a unique left-translation invariant Borel probability measure $\mu_{G/\Gamma}$ called the \emph{Haar measure} on $G/\Gamma$. This allows us to associate to every topological nilsystem $(G/\Gamma, R_a)$ a natural measure-preserving system $(G/\Gamma,\mu_{G/\Gamma}, R_a)$. The Koopman-representation of the transformation $R_a$ is the unitary operator (which, in an abuse of notation, we also denote by $R_a$) on $L^2(G/\Gamma,\mu_{G/\Gamma})$ defined by $R_a f = f\circ R_a$. We say that the nilsystem is \emph{ergodic} if the only $R_a$-invariant functions in $L^2(G/\Gamma,\mu_{G/\Gamma})$ are the constant functions. Due to their rich underlying algebraic structure, the dynamical behaviour of nilsystems is remarkably complaisant:

    \begin{prop}[cf.~\cite{Parry-ergodic_properties}, Section 2] \label{prop: equivalences}
        Let $G/\Gamma$ be a nilmanifold and $a \in G$. The following are equivalent:
        \begin{enumerate}
            \item The topological system $(G/\Gamma, R_a)$ is transitive, i.e., there exists a point with dense orbit.
            \item The topological system $(G/\Gamma, R_a)$ is minimal, i.e., every point has a dense orbit.
            \item The topological system $(G/\Gamma, R_a)$ is uniquely ergodic with unique $G$-invariant measure $\mu_{G/\Gamma}$.
            \item The measure-preserving system $(G/\Gamma,\mu_{G/\Gamma}, R_a)$ is ergodic.
        \end{enumerate}
    \end{prop}
Let $S^1$ denote the unit circle in the complex plane. If $\mu$ and $\nu$ are two Borel measures on $S^1$, we say that $\mu$ is \emph{absolutely continuous} with respect to $\nu$ and write $\mu\ll \nu$ if $\nu(A)=0\Rightarrow \mu(A)=0$ for any Borel subset $A\subseteq S^1$. This introduces a natural partial order on the family of Borel measures on $S^1$. We say that $\mu$ and $\nu$ are \emph{equivalent} if $\mu\ll \nu$ and $\nu \ll \mu$ and write $\nu\approx \mu$. The \emph{type} of a Borel measure $\mu$ is defined as the equivalence class of all Borel measures that are equivalent to $\nu$.

\begin{defn}[Spectral measure and spectral type]
\label{def_spectal_measure_spectral_type}
Let $U$ be a unitary operator on a Hilbert space $\Hc$. The \emph{spectral measure} of $h\in \Hc$ is the unique finite Borel measure $\sigma_h$ on $S^1$ satisfying 
\[
\left<U^n h , h \right> = \int_{S^1} t^n d\sigma_h(t),\qquad \forall n\in\mathbb{Z}.
\]
The existence of this measure is guaranteed by Herglotz's theorem (see~\cite[Sections 7.6 and 7.8]{katznelson}). Moreover, there exists a unique finite Borel measure $\sigma$ on $S^1$, called the \emph{maximal spectral type} of the operator $U$, with the property that every spectral measure of an element in $\Hc$ is absolutely continuous with respect to $\sigma$ and conversely, every finite Borel measure absolutely continuous with respect to $\sigma$ is the spectral measure of some element of $\Hc$ \cite{Halmos}.
\end{defn}

Spectral theory provides an important framework for analyzing linear operators and links operator theory to harmonic analysis.
In ergodic theory, the study of the spectral properties of the Koopman representation of a measure-preserving transformation follows a long history, dating back to the foundational works of von Neumann and Koopman in the 1930s (\cite{Koopman}, \cite{VN1}, \cite{VN2}).
%Developed as an early tool for characterizing and distinguishing between ergodic systems, 
It directly relates to important ergodic-theoretic properties of the underlying dynamical system (such as mixing properties or rigidity phenomena), plays an important role in the classification of measure-preserving systems and their joinings%\cite{Weiss}
, and aids the study of the stability and long-term behaviour of the system, which connects to recurrence and convergence problems in ergodic theory.
For more information about the spectral theory of dynamical systems and its applications, we refer the reader to the survey \cite{Lemanczyksurvey}. 
%and its connections to number theory and combinatorics, Diophantine approximation and uniform distribution.

In general, spectral measures are difficult to compute and even for some of the most well-studied systems the maximal spectral type remains unknown. The purpose of this paper is to settle this problem for the class of nilsystems. Nilsystems and some generalizations thereof have become increasingly important in ergodic theory due to their connections to the structure theory of ergodic averages \cite{host2005nonconventional}, \cite{ziegler2007universal}, \cite{shalom1}, \cite{shalom2}, additive combinatorics \cite{hk-book}, number theory \cite{Zieglersurvey}, \cite{FH1}, \cite{FH2}, nilspace theory \cite{candela}, \cite{gmv}, \cite{gmv2}, \cite{gmv3}, and Higher-order Fourier analysis \cite{gtz}, \cite{tao2012higher},\cite{candela-szegedy-inverse}, \cite{jt21-1}, \cite{jst-tdsystems}, \cite{cgss-2023}. Therefore, determining the maximal spectral type of nilsystems finds several applications. To state our main result, we need one more definition.

% Spectral theory provides an important framework for analyzing linear operators and links operator theory to harmonic analysis.
% In ergodic theory, the study of the spectral properties of the Koopman representation of a measure-preserving transformation dates back to the seminal works of von Neumann and and Koopman in the 1930s (\cite{Koopman}, \cite{VN1}, \cite{VN2}), which laid the foundation for modern ergodic theory.
% It provides insight into the eigenvalues and eigenfunctions of the transformation, which are associated to important characteristics of the underlying dynamical system (such as mixing- and higher-order mixing properties). It allows us to identify the frequencies that resonate with the system, connecting ergodic theory with the study of uniform distribution and Diophantine approximation. Moreover, the spectral analysis of a transformation aids the study of the stability and long-term behaviour of the underlying dynamical system, which plays an important role in the study of recurrence and convergence problems in ergodic theory and its connections to number theory and combinatorics.
    
    \begin{defn}\label{infinite Lebesgue spectrum}
     Let $U$ be a unitary operator on a Hilbert space $\Hc$.
     \begin{itemize}
        \item $U$ has \emph{discrete spectrum} if the maximal spectral type is a discrete measure on $S^1$, or equivalently, if the eigenfunctions of $U$ span a dense subspace of $\Hc$.  
        \item $U$ has \emph{infinite Lebesgue spectrum} if $\Hc$ decomposes into a direct sum of infinitely many pairwise orthogonal closed $U$-invariant subspaces, each with maximal spectral type equivalent to the Lebesgue measure on $S^1$.
        \item $U$ has \emph{compact-Lebesgue spectrum} if $\Hc= \Hc_d\oplus \Hc_l$ is an orthogonal decomposition of $\Hc$ into a closed subspace $\Hc_d$ with discrete spectrum and a closed subspace $\Hc_l$ with infinite Lebesgue spectrum.
     \end{itemize}
    \end{defn}
    In this paper we will only consider separable Hilbert spaces. Therefore, infinite Lebesgue spectrum is a synonym of countable Lebesgue spectrum. 

    The main result of this paper is the following:

    \begin{thm}\label{main}
    Let $k\geq 2$, and let $(G/\Gamma,\mu_{G/\Gamma}, R_a)$ be an ergodic $k$-step nilsystem. Then $R_a$ has compact-Lebesgue spectrum on $L^2(G/\Gamma,\mu_{G/\Gamma})$.
    \end{thm}

This result was previously established by \cite{Stepin} under the additional assumption that $G$ is connected and simply connected (see also \cite{Strakov} for history and comments) and by Host, Kra and Maass \cite{HKM} for $2$-step nilsystems ($k=2$) in full generality. We give an alternative proof of the $k=2$ case in \cref{app: 2-step}. 

\begin{rem}\label{skewproduct}
 Nilsystems are closely related to another class of systems called skew products. Yet, the results of Theorem \ref{main} fail for these systems. Consider the system $\mathrm{X} = (\mathbb{T}^2,\mu_{\T^2},T)$  equipped with the Haar measure $\mu_{\T^2}$ and the action $T(x,y) = (x + \alpha , y+\varphi(x))$, where $\alpha$ is irrational and $\varphi:\T\rightarrow \mathbb{R}$ is continuous. When $\alpha$ is a Liouville number, Herman \cite{Herman} constructed an absolutely continuous function $\varphi$ for which $\mathrm{X}$ is ergodic, rigid and is not a group rotation. Since rigid systems are singular, this example contradicts the conclusion of Theorem \ref{main}. More related examples are surveyed in \cite{Lemanczyksurvey}
\end{rem}

As an application of our main theorem, we also recover the following theorem.
\begin{thm}
    Let $k\geq 1$ and let $(G/\Gamma,\mu_{G/\Gamma}, R_a)$ be an ergodic $k$-step nilsystem. Let $f\in L^\infty(G/\Gamma,\mu_{G/\Gamma})$ and suppose that $f$ is orthogonal to the subspace spanned by all eigenfunctions of $R_a$. Then $$\lim_{n\rightarrow\infty} \int_{G/\Gamma} T^n f \cdot \overline{f} d\mu_{G/\Gamma} = 0.$$
\end{thm}
This result was previously established for $k=2$ by Host-Kra-Maass \cite{HKM}, and for general $k$ by Griesmer \cite{Griesmer2}, and recently by Frantizkinakis and Kuca \cite{FranKuca}, using different methods. We provide yet another proof.
\begin{proof}
    By \cref{def_spectal_measure_spectral_type} we have $\int_{G/\Gamma} T^n f \cdot \overline{f} d\mu_{G/\Gamma} = \int_{S^1} t^n d\sigma_f$ where $\sigma_f$ is the spectral measure of $f$. Write $L^2(G/\Gamma,\mu_{G/\Gamma}) = \mathcal{H}_d \oplus \mathcal{H}_l$, where $\mathcal{H}_d$ has discrete spectrum and $\mathcal{H}_l$ has infinite Lebesgue spectrum. The subspace spanned by all eigenfunctions corresponds to $\mathcal{H}_d$. By assumption, $f$ is orthogonal to $\mathcal{H}_d$ and hence belongs to $\mathcal{H}_l$. Therefore $\sigma_f$ is absolutely continuous with respect to Lebesgue. The result now follows by Riemann-Lebesgue lemma.
\end{proof}

\subsection*{Acknowledgements}{}  The first and third author are supported by the National Science Foundation under grant DMS-1926686. We would like to thank Mariusz Lema\'{n}czyk for helpful discussions leading to Remark \ref{skewproduct}.

\section{Relevant results about nilsystems}

In this preparatory section we survey some well known results about nilsystems used in the proof of \cref{main}.

Let $G$ be a Lie group. For $a, b \in G$, the \emph{commutator} of $a$ and $b$ is the element $[a,b] = a^{-1}b^{-1}ab \in G$. For subgroups $H_1, H_2 \le G$, we denote by $[H_1, H_2]$ the group generated by the commutators $[a,b]$ with $a \in H_1$, $b \in H_2$. For nilpotent Lie groups, an important associated object is its lower central series,
\[
G=G_1 \unrhd G_2 \unrhd \ldots \unrhd G_k \unrhd G_{k+1}=\{e\},
\]
defined recursively as $G_1 = G$ and $G_{i+1} = [G, G_i]$ for $i \ge 1$.
Recall that $G_{k+1} = \{e\}$ if and only if $G$ is nilpotent of step $\le k$.

\begin{rem} \label{rem: connected component} The following properties of nilsystems are useful for us.
\begin{itemize}
   \item[(1)] If any of the conditions in \cref{prop: equivalences} hold, then $(G/\Gamma, R_a)$ is isomorphic to a nilsystem $(H/\Lambda, R_b)$ where $H$ is generated by $b \in Y$ and $H^0$, the connected component of the identity in $H$ (see \cite[Theorem 10.5]{host2005nonconventional}). 

    \item[(2)] If $N \le \Gamma$ is normal in $G$, then $G/\Gamma\cong (G/N)/(\Gamma/N)$. Thus, by taking the maximal such $N$ with respect to inclusion, we may assume without loss of generality that $\Gamma$ contains no normal subgroup of $G$.
    \end{itemize}
\end{rem}

\begin{prop} \label{prop: connected}
    Let $G/\Gamma$ be a nilmanifold, let $a \in G$ such that $(G/\Gamma, \mu_{G/\Gamma}, R_a)$ is ergodic, and assume that $G = \bigl\langle G^0, a \bigr\rangle$.
    Then $G_i$ is connected for every $i \ge 2$.
\end{prop}
\begin{proof}
    First observe that $a\not\in G_2$: otherwise $X$ admits a non-ergodic factor $G/G_2\Gamma$, which is a contradiction. It then follows by \cite[Lemma 2.10]{Leibmanpointwise} that $G_i$ is connected for $i \ge 2$.
\end{proof}

\begin{prop} \label{prop: eigenfunctions}
    Fix $k \ge 2$, and
    let $(G/\Gamma, \mu_{G/\Gamma}, R_a)$ be an ergodic $k$-step nilsystem.
    If $f \in L^2(X)$ is an eigenfunction, then $f$ is measurable with respect to the factor $G/G_2\Gamma$.
\end{prop}
\begin{proof}
    This was established by Leibman in \cite{Leibmanpointwise}. 
\end{proof}

\section{Criterion for Lebesgue spectrum}
The following lemma is a corollary of the Peter-Weyl theorem and is due to Parry \cite[Lemma 1]{Parry}.
\begin{lem}
    Let $k\geq 1$ and let $G/\Gamma$ and $a\in G$, be such that $(G/\Gamma,\mu_{G/\Gamma},R_a)$ is an ergodic $k$-step nilsystem and suppose that $\Gamma$ contains no non-trivial normal subgroups of $G$. Then \begin{equation}\label{orthogonaldecomposition}L^2(G/\Gamma,\mu_{G/\Gamma}) = \bigoplus_{\gamma\in \widehat{G_k}} V_\gamma,
    \end{equation} where
    $$V_\gamma = \{f\in L^2(G/\Gamma,\mu_{G/\Gamma} : f(ux) = \gamma(u) f(x) \text { } \forall u\in G_k\}$$ and $\widehat{G_k}$ is the character group of $G_k$.
\end{lem}

Our main tool for establishing Lebesgue spectrum is the following criterion of Parry.
    \begin{prop}[\cite{Parry}, Corollary 2]\label{ParryCriterion}
    Let $G/\Gamma$ and $a \in G$ be such that $(G/\Gamma, \mu_{G/\Gamma}, R_a)$ is an ergodic nilsystem, and assume that $G = \bigl\langle G^0, a \bigr\rangle$ and $\Gamma$ contains no nontrivial normal subgroups of $G$, which is an assumption that we can always make without loss of generality due to \cref{rem: connected component}). Let $G=G_1 \unrhd G_2 \unrhd \ldots \unrhd G_k \unrhd G_{k+1}=\{e\}$ denote the lower central series of $G$ and write $L^2(G/\Gamma, \mu_{G/\Gamma}) = \bigoplus_{\gamma \in \hat{G_k}}{V_{\gamma}}$ as in \eqref{orthogonaldecomposition}.
    Suppose that for every $u\in G_k$ there exists $b_u\in G_{k-1}$ such that $[a,b_u]=u$. Then for any nontrivial $\gamma \in \hat{G_k} \setminus \{1\}$, the maximal spectral type of $(V_\gamma,R_a)$ is Lebesgue.
   \end{prop}
   \begin{proof}
   Let $\gamma$ be a nontrivial character and $f\in V_\gamma$. Let $u\in G_k$ and let $b_u$ be as in the proposition. Then $$ab_u = b_u a u$$ and more generally, for all $n\in \mathbb{Z}$ we have $$a^nb_u = b_ua^n u^n.$$
   For $f\in V_\gamma$ we let $\sigma_f$ denote the spectral measure of $f$ with respect to $R_a$. Since $b_u$ commutes with $G_k$, it maps $V_\gamma$ to itself. We then have
   
  $$\int_{S^1} \lambda^n d\sigma_{b_uf} = \left<a^n b_uf,b_uf\right> =  \left<b_ua^nu^nf,b_uf\right> = \left<a^n u^n f, f\right> = \int_{S^1} \gamma(u)^n\lambda^n d\sigma_f.$$
  That is, $\sigma_{b_u f}$ is equal to the measure $\sigma_f^{\gamma(u)}$ obtained by the change of variables $\lambda \mapsto \gamma(u) \lambda$. Let $\sigma$ be the maximal spectral type of $V_\gamma$. In view of \cref{prop: connected} the set $\gamma(G_k)$ is a connected subgroup of $S^1$. Since $\gamma \in \hat{G_k} \setminus \{1\}$ we conclude that $\gamma(G_k)\neq \{1\}$ and hence $\gamma(G_k)=S^1$. So
  if $\nu\ll \sigma$, then $\nu^s \ll \sigma$ for all $s\in S^1$. We claim from this observation that $\sigma$ is Lebesgue. First note that $\int_{S^1} \sigma^t dt$ is a rotation-invariant measure on $S^1$, and so it must be equal to Lebesgue. Therefore, the Lebesgue measure is absolutely continuous with respect to $\sigma$. For the other direction, let $A$ be a set with Lebesgue measure zero and assume for contradiction that $\sigma(A)\not = 0$. By definition it follows that $\sigma^t(t^{-1}A)> 0$, but since $\sigma^t$ is absolutely continuous with respect to $\sigma$ we get $\sigma(t^{-1}A)> 0 $ for all $t\in S^1$. Therefore, $\int_{S^1} \sigma(t^{-1}A) dt >0$, but the measure defined by this integral is Lebesgue which leads to a contradiction.
   \end{proof}

\section{Proof of the main result} \label{sec: proof}

\begin{proof}[Proof of \cref{main}]
The $k=2$ case is proved in \cite{HKM} (see also \cref{app: 2-step} below), so assume $k > 2$ and $(G/\Gamma, \mu_{G/\Gamma}, R_a)$ is an ergodic $k$-step nilsystem.
By \cref{rem: connected component}, we may assume $G = \bigl\langle G^0, a \bigr\rangle$ and $\Gamma$ does not contain any normal subgroup of $G$. In particular, $G_k$ is compact and connected, and $G_k \cap \Gamma = \{e\}$.

By \cref{prop: eigenfunctions}, and induction on the degree of nilpotency of $G$, it suffices to show
\[
L^2(G/\Gamma, \mu_{G/\Gamma}) \cong L^2(G/G_k\Gamma,\mu_{G/G_k\Gamma}) \oplus V
\]
where $V$ is a closed and invariant subspace of $L^2(G/\Gamma, \mu_{G/\Gamma})$ on which $R_a$ has infinite Lebesgue spectrum.
If $G$ is $(k-1)$-step nilpotent, then $V=0$ and there is nothing to prove. We therefore assume that this is not the case.
As in the proof of \cref{prop: eigenfunctions}, the space $V$ decomposes as $V = \bigoplus_{\gamma} V_\gamma$ with $V_\gamma = \left\{ v \in V : u \cdot v = \gamma(u)v~\text{for all}~u \in G_k \right\}$. Taking a quotient by $\ker{\gamma}$ for $\gamma \in \hat{G_k}$, we may assume that $G_k$ is a subgroup of $S^1$. Since $G_k$ is non-trivial and connected, we then actually have $G_k = S^1$.

We claim that $A:=\{[a,g] : g\in G_{k-1}\} = G_k$. Since $g\mapsto [a,g]$ is continuous, and $G_{k-1}$ is connected (by \cref{prop: connected}), we have that $A$ is connected.
Moreover, $x \mapsto [a, x]$ is a homomorphism from $G_{k-1}$ to $G_k$. Indeed,
\begin{equation}
    [a,xy] = a^{-1} y^{-1} x^{-1} a x y = a^{-1} y^{-1} (a a^{-1}) x^{-1} a x y = a^{-1} y^{-1} a [a,x] y = [a,x][a,y].
\end{equation}
Therefore, $A$ is a connected subgroup of $S^1$ and so must be all of $S^1$ or trivial.

Suppose for contradiction that $A = \{e\}$. We claim that this implies $[G, G_{k-1}] = \{e\}$, or equivalently that $G$ is $(k-1)$-step nilpotent which contradicts the assumption. Since $G = \bigl\langle G^0, a \bigr\rangle$, it suffices to show $[G^0, G_{k-1}] = \{e\}$. Fix $y \in G$. Since $(G/\Gamma, R_a)$ is minimal (cf.~\cref{prop: equivalences}), we may find sequences $\gamma_n \in \Gamma$ and $t_n \in \Z$ such that $a^{t_n}\gamma_n \to y$ in $G$. Let $\gamma \in G_{k-1} \cap \Gamma$.
Then, on the one hand, $\gamma a^{t_n}\gamma_n \to \gamma y$. On the other hand, $\gamma a^{t_n} \gamma_n = a^{t_n} \gamma \gamma_n=a^{t_n}\gamma_n \gamma [\gamma,\gamma_n]$. But since $[\gamma,\gamma_n]\in \Gamma\cap G_k =\{e\}$, we have $a^{t_n} \gamma_n \gamma \to y\gamma$.
Hence, $\gamma y = y \gamma$. Thus,
\begin{equation} \label{eq: trivial commutator}
    [G, G_{k-1} \cap \Gamma] = \{e\}.
\end{equation}
Now consider the map $s \mapsto [s, \cdot]$ from $G^0$ to $\mathrm{Hom}(G_{k-1}, G_k)$. By \eqref{eq: trivial commutator}, the homomorphism $[s, \cdot]$ is $(G_{k-1} \cap \Gamma)$-invariant for each $s \in G^0$. It is also $G_k$-invariant. Hence, $[s, \cdot]$ descends to a homomorphism from $H = G_{k-1}/G_k(G_{k-1} \cap \Gamma)$ to $G_k = S^1$. The group $H$ is a compact abelian group, so $\mathrm{Hom}(H, G_k) = \hat{H}$ is discrete. But since $G^0$ is connected, $s\mapsto [s,\cdot]$ is continuous, and $[e, \cdot]$ is trivial, we conclude $[G^0, G_{k-1}] = \{e\}$, yielding a contradiction. 

We deduce that $A=G_k$. In this case the assumption in \cref{ParryCriterion} is satisfied and so the maximal spectral type of $V_\gamma$ is Lebesgue. Since $G_k$ is connected it has infinitely many non-trivial characters, showing that the Lebesgue spectrum has infinite multiplicity. This completes the proof of \cref{main}.
\end{proof}

\appendix

\section{Maximal spectral type of $2$-step nilsystems} \label{app: 2-step}

In this appendix, we present a short proof of \cref{main} for 2-step nilsystems. The main fact which enables us to produce compact-Lebesgue spectrum in this case is that compact-Lebesgue spectrum is preserved by relatively independent joinings:

\begin{prop} \label{prop: joinings}
	Let $(X, \mu, T)$ and $(Y, \nu, S)$ be measure-preserving systems with compact-Lebesgue spectrum.
	Then the relatively independent joining $(X \times Y, \mu \times_Z \nu, T \times S)$ over any common factor $Z$
	also has compact-Lebesgue spectrum.
\end{prop}
\begin{proof}
	Let $\pi_1 : X \to Z$ and $\pi_2 : Y \to Z$ be factor maps.
	Let $K_T$ and $K_S$ be the Kronecker factors of $(X, \mu, T)$ and $(Y, \nu, S)$ respectively.
	Note that the meet $K_T \land Z$ is equal to the Kronecker factor of $Z$, which we denote by $K_Z$.
	Similarly, $K_S \land Z = K_Z$.
	
	We have the following splittings into invariant subspaces
	\begin{align*}
		L^2(X, \mu) & = L^2(K_Z) \oplus U \oplus V_1 \oplus V_2
		\intertext{and}
		L^2(Y, \nu) & = L^2(K_Z) \oplus U \oplus W_1 \oplus W_2,
	\end{align*}
	where
	\begin{align*}
		L^2(K_Z) \oplus U & = L^2(Z), \\
		L^2(K_Z) \oplus V_1 & = L^2(K_T), \\
		L^2(K_Z) \oplus W_1 & = L^2(K_S).
	\end{align*}
	By assumption, $\left. T \right|_{U \oplus V_2}$ and $\left. S \right|_{U \oplus W_2}$ have Lebesgue spectrum.
	
	We claim
	\begin{align*}
		L^2(X \times Y, \mu \times_Z \nu) = L^2(Z) \oplus V \oplus W \oplus (V \otimes W),
	\end{align*}
	where $V = V_1 \oplus V_2$, $W = W_1 \oplus W_2$, and $V \otimes W$ is the Hilbert space tensor product of $V$ and $W$.
	Indeed, suppose $f \in L^2(X \times Y, \mu \times_Z \nu)$.
	We may assume without loss of generality that $f(x,y) = g(x) h(y)$ for some functions $g \in L^2(X, \mu)$, $h \in L^2(Y, \nu)$.
	Using the splittings above, we may write
	\begin{align*}
		g = \tilde{g} \circ \pi_1 + v \qquad \text{and} \qquad h = \tilde{h} \circ \pi_2 + w,
	\end{align*}
	where $\tilde{g} = (\pi_1)_*g$, $\tilde{h} = (\pi_2)_*h$, $v \in V$, and $w \in W$.
	Since $\pi_1(x) = \pi_2(y)$ for $(\mu \times_Z \nu)$-a.e. $(x,y) \in X \times Y$, we have
	\begin{align*}
		f(x,y) = \tilde{g}(z) \tilde{h}(z) + v(x) \tilde{h}(z) + \tilde{g}(z) w(y) + v(x) w(y),
	\end{align*}
	where $z = \pi_1(x) = \pi_2(y)$.
	This decomposition consists of a sum of elements from $L^2(Z)$, $V$, $W$, and $V \otimes W$ as desired.
	
	Decomposing $L^2(Z)$, $V$, and $W$ further, the space $L^2(X \times Y, \mu \times_Z \nu)$ is equal to
	\begin{align*}
		L^2(K_Z) \oplus U \oplus V_1 \oplus V_2
		 \oplus W_1 \oplus W_2 \oplus (V_1 \otimes W_1) \oplus (V_1 \otimes W_2)
		 \oplus (V_2 \otimes W_1) \oplus (V_2 \otimes W_2).
	\end{align*}
	
	It is easily checked that $T \times S$ has discrete spectrum on the subspace
	\begin{align*}
		L^2(K_Z) \oplus V_1 \oplus W_1 \oplus (V_1 \otimes W_1).
	\end{align*}
	Moreover, on the orthogonal complement
	\begin{align*}
		U \oplus V_2 \oplus W_2 \oplus (V_1 \otimes W_2) \oplus (V_2 \otimes W_1) \oplus (V_2 \otimes W_2),
	\end{align*}
	$T \times S$ has Lebesgue spectrum.
	This is clear for the subspaces $U$, $V_2$, and $W_2$.
	For the remaining subspaces, note that for $v \in V$ and $w \in W$, one has $\sigma_{v \otimes w} = \sigma_v * \sigma_w$.
	Hence, if the spectral measure associated to either $v$ or $w$ is the Lebesgue measure,
	then $\sigma_{v \otimes w}$ is also the Lebesgue measure.
\end{proof}

The following description of the structure of 2-step nilsystems allows us, with the help of \cref{prop: joinings}, to reduce to the case that either $G$ is connected or the niltranslation is isomorphic to an affine transformation on a finite-dimensional torus.

\begin{thm} \label{thm: main thm}
	Let $G/\Gamma$ be a connected $2$-step nilmanifold, and suppose $(G/\Gamma, R_a)$ is a minimal nilsystem. Let $G^0$ denote the conncected component of the identity in $G$ and $\Gamma^0=\Gamma\cap G^0$.
	Then there exist $b \in G^0$, $d \in \N$, and a minimal unipotent affine transformation $S \colon \T^{2d} \to \T^{2d}$ such that
	$(G/\Gamma, R_a)$ is a factor of a relatively independent joining of $(G^0/\Gamma^0, R_b)$ with $(\T^{2d}, S)$.
\end{thm}
\begin{proof}
    Since $G/\Gamma$ is connected, there exists $\gamma \in \Gamma$ so that $b = a \gamma \in G^0$.
Then for any $x = g \Gamma \in X$, we have $ax = b\gamma^{-1}g\Gamma = bg[g, \gamma]\Gamma$.
Since $G$ is a 2-step nilpotent group, the commutator $[g, \gamma]$ belongs to the center of $G$, so $ax = [g, \gamma] bx$.
Moreover, one can check that if $g \equiv h \pmod{\Gamma}$, then $[g, \gamma] \equiv [h, \gamma] \pmod{\Gamma}$, so we may write $ax = [x, \gamma] bx$ without any ambiguity.

If $[x, \gamma] = \Gamma$ for every $x \in X$, then $R_a = R_b$, so there is nothing to prove.
Assume $[x, \gamma] \ne \Gamma$ for some $x \in X$.
Denote by $[X, \gamma]$ the set $\{[x, \gamma] : x \in X\} \subseteq G_2/(G_2 \cap \Gamma)$.
Then $[X, \gamma]$ is a connected abelian Lie group, so there is an isomorphism $\varphi : [X, \gamma] \to \T^d$ for some $d \in \N$.
Define $S : \T^{2d} \to \T^{2d}$ by $S(u,v) = (u+\alpha, v+u)$, where $\alpha = \varphi([b, \gamma])$.
The map $(u,v) \mapsto u$ is clearly a factor map $(\T^{2d}, S) \to (\T^d, R_{\alpha})$.
Moreover, the map $x \mapsto \varphi([x, \gamma])$ is a factor map $(X, R_b) \to (\T^d, R_{\alpha})$.
Indeed, for any $x \in X$,
\begin{align*}
	\varphi([bx, \gamma]) = \varphi([b, \gamma] [x, \gamma]) = \varphi([x, \gamma]) + \alpha.
\end{align*}

The relatively independent joining of $(X, R_b)$ and $(\T^{2d}, S)$ over the factor $(\T^d, R_{\alpha})$ is isomorphic to the system $(Z, T)$, where $Z = X \times \T^d$ and
\begin{align*}
	T(x,v) = (bx, v + \varphi([x, \gamma])).
\end{align*}

We claim $(X, R_a)$ is a factor of $(Z, T)$.
Define $\pi : Z \to X$ by $\pi(x,v) = \varphi^{-1}(v)x$.
Then
\begin{align*}
	\pi(T(x,v)) = \pi(bx, v + \varphi([x, \gamma]))
	 = \varphi^{-1}(v) [x, \gamma] bx = \varphi^{-1}(v) ax = a \varphi^{-1}(v)x = R_a \pi(x,v).
\end{align*}
Moreover, $\pi(x,1) = x$, so $\pi$ is surjective.
Thus, $\pi$ is a factor map $(Z, T) \to (X, R_a)$. Since $(X,R_b)$ and $(G^0/\Gamma^0,R_b)$ are isomorphic, the proof is complete.
\end{proof}

%A unipotent affine transformation $S:\mathbb{T}^{2d}\rightarrow \mathbb{T}^{2d}$ is a special case of system called Weyl systems (see \cite[Definition 1.5]{btz2}).
%\begin{rem}
It is natural to ask what would be a version of Theorem \ref{thm: main thm} for higher values of $k$. The system $(\T^{2d}, S)$ appearing in the conclusion of Theorem \ref{thm: main thm} belongs to a class of systems called Weyl systems introduced in \cite{bll}. A \emph{Weyl system} $(Y,S)$ consists of a compact abelian Lie group $Y$ and a unipotent affine transformation $S : Y \to Y$. Weyl systems form a subclass of the class of nilsystems and enjoy many nice dynamical properties (see \cite[Section 3]{bll}). %We believe that a version of Theorem \ref{thm: main thm} for higher values of $k$ may involve these $k$-step nilsystems.
%\end{rem}
\begin{question}
    Let $G/\Gamma$ be a connected $k$-step nilmanifold, and suppose that $(G/\Gamma,R_a)$ is a minimal nilsystem. Let $G^0$ denote the connected component of the identity in $G$ and $\Gamma^0=\Gamma\cap G^0$. Do there exists $b\in G^0$ and a $k$-step Weyl system $(Y,S)$ such that $(G/\Gamma,R_a)$ is a factor of a relatively independent joining of $(G^0/\Gamma^0,R_b)$ with $(Y,S)$?
\end{question}
\begin{comment}
The system $(\T^{2d}, S)$ appearing in the conclusion of Theorem \ref{thm: main thm} belongs to a class of systems called Weyl systems introduced in \cite{bll}. A \emph{Weyl system} $(Y,S)$ consists of a compact abelian Lie group $Y$ and a unipotent affine transformation $S : Y \to Y$. Weyl systems form a subclass of the class of nilsystems and enjoy many nice dynamical properties (see \cite[Section 3]{bll}). 
We believe that Theorem \ref{thm: main thm} can be generalized to higher values of $k$ as follows:
\begin{conj}
Let $G/\Gamma$ be a connected $k$-step nilmanifold, and suppose that $(G/\Gamma,R_a)$ is a minimal nilsystem. Let $G^0$ denote the connected component of the identity in $G$ and $\Gamma^0=\Gamma\cap G^0$. Then there exists $b\in G^0$, and a $k$-step Weyl system $(Y,S)$ such that $(G/\Gamma,R_a)$ is a factor of a relatively independent joining of $(G^0/\Gamma^0,R_b)$ with $(Y,S)$.
\end{conj}
\end{comment}

We can now prove \cref{main} when $k=2$.
We will first carry out the proof in the case that $X = G/\Gamma$ is a connected 2-step nilmanifold and then deduce the general case from this one.
By \cref{thm: main thm}, $(X, \mu_x, R_a)$ is a factor of a relatively independent joining of an ergodic nilsystem $(X, \mu_x, R_b)$ with $b \in G^0$ and an ergodic affine transformation $(\T^{2d}, \mu_{\T^{2d}}, S)$ for some $d \in \N$.
A factor of a system with compact-Lebesgue spectrum clearly has compact-Lebesgue spectrum, so it suffices to prove that the relatively independent joining of $(X, \mu_x, R_b)$ with $(\T^{2d}, \mu_{\T^{2d}}, S)$ has compact-Lebesgue spectrum. By \cref{prop: joinings}, we may further reduce to showing that each of the systems $(X, \mu_X, R_b)$ and $(\T^{2d}, \mu_{\T^{2d}}, S)$ has compact-Lebesgue spectrum.

To show that $(X, \mu_X, R_b)$ has compact-Lebesgue spectrum, we may assume $X = G/\Gamma$ with $G$ connected by \cref{rem: connected component}, since $b \in G^0$. In this case, one may follow the proof in \cref{sec: proof} as written, since $G_{k-1} = G$ is now a connected group.

The fact that $(\T^{2d}, \mu_{\T^{2d}}, S)$ has compact-Lebesgue spectrum follows from a straightforward calculation. We give the details for $S$ of the form $S(u,v) = (u + \alpha, v + u)$ (as appears in the proof of \cref{thm: main thm}) for completeness.
First, $L^2(\T^{2d})$ splits as $U \oplus V$, where $U \cong L^2(\T^d)$ consists of functions $f(u,v)$ that depend only on $u$.
The transformation $S$ is compact on $U$ with eigenvalues $n \cdot \alpha$, $n \in \Z^d$, corresponding to the eigenfunctions $e_{n,0}(u,v) = e(n \cdot u)$.
Now $V$ is spanned by the functions $e_{n,m}(u,v) = e(n \cdot u + m \cdot v)$ with $n, m \in \Z^d$, $m \ne 0$.
By direct calculation, the spectral measure $\sigma_{n,m}$ of $e_{n,m}$ has Fourier coefficients
\begin{align*}
    \hat{\sigma_{n,m}}(k) = \innprod{S^k e_{n,m}}{e_{n,m}}
     & = \int_{\T^d \times \T^d}{e_{n,m} \left( u+k\alpha, v + ku + \binom{k}{2}\alpha \right) \overline{e_{n,m}(u,v)}~du~dv} \\
    & = e \left( \left( kn + \binom{k}{2} m  \right) \cdot \alpha \right) \int_{\T^d}{e(k m \cdot u)~du} = 0
\end{align*}
for $k \ne 0$, since $km \in \Z^d \setminus \{0\}$ gives rise to a nontrival character on $\T^d$.
Hence, $\sigma_{n,m}$ is equal to the Lebesgue measure on $S^1$. \\

Now suppose $X = G/\Gamma$ is a (not necessarily connected) 2-step nilmanifold and $a \in G$ such that $(X, R_a)$ is minimal. Then $X$ has finitely many connected components, say $X = \bigcup_{i=1}^l{X_i}$, and $R_a^k(X_i) \subseteq X_i$ for $k = l! \in \N$. Let $f \in L^2(X)$ with $f$ orthogonal to the Kronecker factor of $(X, \mu_X, R_a)$. Then $f$ is orthogonal to the Kronecker factor on each ergodic component $(X_i, \mu_{X_i}, R_a^k)$ for the transformation $R_a^k$. The spectral measure of $f$ with respect to the transformation $R_a^k$ has Fourier coefficients
\begin{align*}
    \hat{\sigma_{f; R_a^k}}(n) = \int_X{f(a^{kn}x) \overline{f(x)}~d\mu_X}
     = \frac{1}{l} \sum_{i=1}^l{\int_{X_i}{f_i(a^{kn}x) \overline{f_i(x)}~d\mu_{X_i}}},
\end{align*}
where $f_i = f|_{X_i}$.
Since each of the systems $(X_i, \mu_{X_i}, R_a^{k})$ is an ergodic 2-step nilsystem on a connected nilmanifold $X_i$, the sequence
\begin{align*}
    c_i(n) = \int_{X_i}{f_i(a^{kn}x) \overline{f_i(x)}~d\mu_{X_i}}
\end{align*}
is the Fourier transform of a function $\varphi_i \in L^1(S^1)$. Therefore, $\sigma_{f; R_a^k}$ is absolutely continuous with respect to the Lebesgue measure on $S^1$, with Radon--Nikodym derivative $\varphi = \frac{1}{l} \sum_{i=1}^l{\varphi_i}$.

Now, for $A \subseteq S^1$, $\sigma_{f; R_a^k}(A) = \sigma_{f; R_a}(\{z \in S^1 : z^k \in A\})$. Hence, $\sigma_{f; R_a}(A) \le \sigma_{f; R_a^k}(\{z^k : z \in A\})$, so $\sigma_{f; R_a}$ is also absolutely continuous with respect to the Lebesgue measure on $S^1$. The following observation then completes the proof:

\begin{lem}
	Let $(X, \mu, T)$ be an ergodic measure-preserving system with an irrational eigenvalue.
	Let $\sigma$ be the maximal spectral type of $(X, \mu, T)$.
	Decompose $\sigma = \sigma_d + \sigma_s + \sigma_{ac}$ as a sum of
	discrete, singular, and absolutely continuous measures.
	Then $\sigma_{ac}$ is equivalent to Lebesgue measure.
\end{lem}
\begin{proof}
	Let $f \in L^2(X)$ with $\sigma_f \approx \sigma_{ac}$.
	Let $g : X \to S^1$ be an eigenfunction $Tg = e(\alpha)g$ with $\alpha \notin \Q$.
	Then
	\begin{align*}
		\hat{\sigma}_{fg}(n) = \innprod{T^n(fg)}{fg}
		 = \innprod{e(n\alpha)g T^nf}{fg}
		 = e(n\alpha) \innprod{T^nf}{f}
		 = e(n\alpha) \hat{\sigma}_f(n).
	\end{align*}
	Hence, $\sigma_{fg}(A) = \sigma_f(e(\alpha)A)$.
	Define a measure
	\begin{align*}
		\nu(A) = \int_{S^1}{\sigma_f(tA)~dt}.
	\end{align*}
	By construction, $\nu$ is translation-invariant, so $\nu$ is a nonzero scalar multiple of Lebesgue measure.
	We want to show $\nu \ll \sigma_f$.
	Suppose $\sigma_f(A) = 0$.
	Since $\sigma_{ac} \ll \sigma_f$, we have $\sigma_{fg^n} \ll \sigma_f$ for every $n \in \Z$.
	Hence $\sigma_f(e(n\alpha) A) = 0$ for all $n \in \Z$.
	But $t \mapsto \sigma_f(tA)$ is a continuous function and $\{e(n\alpha) : n \in \Z\}$ is dense in $S^1$,
	so $\sigma_f(tA) = 0$ for every $t \in S^1$.
	By the definition of $\nu$, it follows that $\nu(A) = 0$.
	That is, $\nu \ll \sigma_f$ as claimed.
\end{proof}

\bibliographystyle{abbrv}
\bibliography{bibliography}
\end{document}